\documentclass{custom}

\usepackage{amsmath}
\usepackage{amssymb}

\usepackage[all,2cell,cmtip]{xy}

\def\G{{\mathbb G}}

\def\P{{\mathbb P}}

\def\cC{{\mathcal C}}
\def\cD{{\mathcal D}}
\def\cE{{\mathcal E}}
\def\cF{{\mathcal F}}

\def\cM{{\mathcal M}}
\def\cN{{\mathcal{N}}}
\def\cO{{\mathcal{O}}}
\def\cP{{\mathcal{P}}}

\def\cU{{\mathcal U}}

\def\cX{{\mathbb P(T_X)}}
\def\cY{{\mathbb Y}}

\def\G{{\mathbb{G}}}

\newcommand{\ol}[1]{\overline{#1}}

\def\fg{{\mathfrak g}}

\def\fsl{{\mathfrak{sl}}}

\def\sl{\operatorname{SL}}

\def\lra{\longrightarrow}

\def\lra{\longrightarrow}
\def\rat{\dashrightarrow}

\def\operatorname#1{\mathop{\rm #1}\nolimits}

\def\Aut{\operatorname{Aut}}

\def\Exc{\operatorname{Exc}}

\def\Hom{\operatorname{Hom}}

\def\Spec{\operatorname{Spec}}

\def\codim{\operatorname{codim}}

\def\im{\operatorname{Im}}

\def\rat{\operatorname{RatCurves}}

\def\ev{{\operatorname{ev}}}

\makeatletter
 \newcommand*\tl[1]{\mathpalette\wthelper{#1}}
\newcommand*\wthelper[2]{%
        \hbox{\dimen@\accentfontxheight#1%
                \accentfontxheight#11.15\dimen@
                $\m@th#1\widetilde{#2}$%
                \accentfontxheight#1\dimen@
        }%
}

\newcommand*\accentfontxheight[1]{%
        \fontdimen5\ifx#1\displaystyle
                \textfont
        \else\ifx#1\textstyle
                \textfont
        \else\ifx#1\scriptstyle
                \scriptfont
        \else
                \scriptscriptfont
        \fi\fi\fi3
}
\makeatother
\begin{document}


\title{Uniform families of minimal rational curves on Fano manifolds \thanks{First author partially
 supported by PRIN project ``Geometria delle variet\`a algebriche'' and the Department of Mathematics of the University of Trento. Second author supported by the Korean National Researcher Program 2010-0020413 of NRF, and by the Polish National Science Center project 2013/08/A/ST1/00804. Third author partially supported by JSPS KAKENHI Grant Number 26800002.}}



\author{Gianluca Occhetta \and Luis E. Sol\'a Conde \and Kiwamu Watanabe}\thanks{sdsdfsdfsdfsdf}


\authorrunning{G. Occhetta \and L. E. Sol\'a Conde \and K. Watanabe} 

\institute{G. Occhetta \at
            Dipartimento di Matematica, Universit\`a di Trento, via
Sommarive 14 I-38123 Povo di Trento (TN), Italy \\
         \email{gianluca.occhetta@unitn.it}         
           \and
           L.E. Sol\'a Conde \at
            Dipartimento di Matematica, Universit\`a di Trento, via
Sommarive 14 I-38123 Povo di Trento (TN), Italy \\
\email{lesolac@gmail.com}
\and K. Watanabe \at 
Course of Mathematics, Programs in Mathematics, Electronics and Informatics,\\
Graduate School of Science and Engineering, Saitama University.\\
Shimo-Okubo 255, Sakura-ku Saitama-shi, 338-8570 Japan\\
\email{kwatanab@rimath.saitama-u.ac.jp}
}


\maketitle

\begin{abstract}
It is a well known fact that families of minimal rational curves on rational homogeneous manifolds of Picard number one are uniform, in the sense that the tangent bundle to the manifold has the same splitting type on each curve of the family. In this note we prove that certain --stronger-- uniformity conditions on a family of minimal rational curves on a Fano manifold of Picard number one allow to prove that the manifold is homogeneous.
\end{abstract}


\maketitle

\section{Introduction}\label{sec:intro}
In the framework of higher dimensional complex algebraic geometry, rational homogeneous manifolds constitute one of the most important classes of examples of Fano manifolds. Their geometric properties may be written in the language of representation theory of complex reductive Lie groups, which makes this class one of the best understood. One of the most important questions in this context was inspired by Mori and posed by Campana and Peternell: Can rational homogeneity be described, within the class of Fano varieties, in terms of positivity properties of the tangent bundle? Namely, does the nefness of $T_X$ imply that a Fano manifold $X$ is homogeneous?

The strategy towards a solution of the Campana-Peternell problem that we have been considering recently (see \cite{MOSW}, \cite{MOSWW} and  \cite{OSWW}) is based on reconstructing the rational homogeneous structure upon families of minimal rational curves contained in the manifold $X$. Philosophically speaking, rational curves on Fano manifolds play a role as central as the one of $\fsl(2)$ in the representation theory of reductive groups, and this analogy should become obvious in our candidates to be rational homogeneous manifolds. 

One of the fundamental problems in this approach is to understand how nefness is reflected on properties of certain families of rational curves in $X$. For instance, one may pose the following question, which is still unanswered, to our best knowledge:

\begin{question}\label{quest:uniform}
Let $\cM$ be a locally unsplit dominating family of rational curves on a Fano manifold $X$ with nef tangent bundle. Is $\cM$ uniform or, at least, locally uniform? (see Definition \ref{def:uniform}).
\end{question}

Associated with $\cM$ one may consider the family of minimal sections of the Grothendieck projectivization of the tangent bundle $T_X$, denoted by $\P(T_X)$, over curves of the family $\cM$ (see Definition \ref{def:minsec2}), that we denote by $\ol{p}:\ol{\cU}\to\ol{\cM}$, and study the different splitting types of the tangent bundle of  $\P(T_X)$ with respect to curves of $\ol{\cM}$. Looking at the rational homogeneous examples, one may check that the family $\ol{\cM}$ is, in general, not uniform, and the different splitting types are related to the singular locus stratification of the crepant contraction of $\P(T_X)$ associated with the line bundle $\cO_{\P(T_X)}(1)$ --note that it is not yet clear whether the nefness assumption on the tangent bundle of a Fano manifold implies its semiampleness. This singular locus stratification is related, for a rational homogeneous manifold, to the orbit decomposition of the corresponding Lie algebra $\fg=H^0(X,T_X)$. In fact, the image of every stratum into $\P(\fg^\vee)$ is the closure of (the set of classes modulo homotheties of) a nilpotent orbit (see \cite{CoMc} for an account on these orbits). 

In this paper we analyze the simplest case, in which the families $\cM$ and $\ol{\cM}$ are both uniform (we simply say that $\cM$ is $2$-uniform, see Definition \ref{def:2-uniform}). 
The main result  shows that, up to a technical assumption, this property is only fulfilled by rational homogenous manifolds:

\begin{theorem}\label{thm:main}
Let $X$ be a Fano manifold of Picard number one 
supporting an unsplit $2$-uniform dominating family $\cM$ of rational curves. Assume moreover that the anticanonical degree  $-K_X\cdot\cM$ of the family is smaller than or equal to $2(\dim(X)+2)/3$.
Then $X=G/P$, where $G$ is a semisimple Lie group with Dynkin diagram $\cD$, $P$ is the parabolic subgroup associated to the $i$-th node of the diagram,  and the pair $(\cD,i)$ is one of the following:
$$
({\rm A}_{1},1),\quad({\rm A}_{k+1},2),\,\,\,k\geq 2,\quad
({\rm B}_2,1),\quad ({\rm D}_{5},5),\quad
({\rm E}_{6},1).
$$
\end{theorem}

In the above statement the cited rational homogeneous manifolds are described in terms of their marked Dynkin diagrams (where the nodes have been numbered as in \cite[p. 58]{Hum}). 

Note that the number $-K_X\cdot\cM-2$ equals the dimension of the {\it variety of minimal rational tangents of $\cM$ at $x$}, denoted by $\cC_x\subset\P(\Omega_{X,x})$. In the range $-K_X\cdot\cM>2(\dim(X)+2)/3$ one could expect, via a positive answer to Hartshorne's conjecture in our particular situation, that $\cC_x$ is a complete intersection in $\P(\Omega_{X,x})$. If this is case, we may conclude the following:

\begin{corollary}\label{cor:main}
Let $X$ be a Fano manifold of Picard number one, not isomorphic to a projective space, supporting an unsplit $2$-uniform dominating family $\cM$ of rational curves, and assume that, for the general point $x$, $\cC_x$ is a complete intersection. Then $X$ is isomorphic to a smooth quadric. 
\end{corollary}

Putting this back into the context of the Campana--Peternell Conjecture and Question \ref{quest:uniform}, let us assume that $X$ is a manifold of Picard number one such that $T_X$ is nef and big, and let $\epsilon$ be the birational crepant contraction of $\P(T_X)$ associated to $\cO_{\P(T_X)}(1)$. It was proved in \cite{1-ample} that if the restriction of $\epsilon$ to $\Exc(\epsilon)$ has $1$-dimensional fibers, then necessarily $X$ is a smooth hyperquadric. Our main result in this paper implies the following extension of \cite{1-ample}:

\begin{corollary}\label{cor:e-ample}
Let $X$ be a Fano manifold of Picard number one, with nef and big tangent bundle, and assume that, with the same notation as above, the restriction of $\epsilon$ to $\Exc(\epsilon)$ is a smooth morphism. Assume moreover that $X$ supports an unsplit uniform dominating family of rational curves of anticanonical degree smaller than or equal to $2(\dim(X)+2)/3$. Then $X$ is rational homogeneous. 
\end{corollary}

The structure of the paper is the following: Sections  \ref{sec:nodal} and \ref{sec:ratcurves} contain some preliminary material on  dual varieties and rational curves on projective varieties, respectively. Section \ref{sec:minsec} deals with the family of minimal sections of $\P(T_X)$ over curves of a family $\cM$, and the duality relation of its locus with the variety of minimal rational tangents of $\cM$. Finally Section \ref{sec:proofs} contains the proofs of the main results of the paper.

%
%


\section{Preliminaries on duals of varieties with nodal singularities}\label{sec:nodal}

We will recall here some standard facts about dual varieties that we will need later on, focusing in the concrete setting we are interested in, that is, in the context of varieties with only {\it nodal singularities}. Though these results are probably well known, they are usually stated in the context of smooth varieties (see  \cite{Ein1} or  \cite[Chapters 1,4,7]{Tev}), so we briefly present them here for the reader's convenience.\par
\medskip
\noindent{\bf Notation}
Throughout the paper we will work over the field of  complex numbers. By $\P(\cE)$ we will denote the Grothendieck projectivization of a complex vector space $\cE$, or of a vector bundle $\cE$ over a certain scheme that should be understood from the context. Along the paper we will often consider vector bundles on the projective line $\P^1$. For simplicity, we will denote by $E(a_1^{k_1},\dots,a_r^{k_r})$ the vector bundle $\bigoplus_{j=1}^r\cO(a_j)^{\oplus k_j}$ on $\P^1$.

\begin{definition}\label{def:dual}
Let $C\subset\P^r=\P(V)$ be an irreducible projective variety. Denoting by $C_0\subset C$ its subset of smooth points, the Euler sequence provides a surjection $\cO_{C_0}\otimes V^\vee\to \cN_{C_0,\P^r}(-1)$, so that we have a morphism $p_2:\P(\cN_{C_0,\P^r}(-1))\to \P(V^\vee)$. Then the closure of the image $C^\vee$ is called the {\it dual variety of} $C$. 
\end{definition}

In other words, $C^\vee$ may be described as the closure of the set of tangent hyperplanes of $C$. That is, we may consider $\P(\cN_{C_0,\P^r}(-1))$ as a subset of $\P(T_{\P^r})\subset\P^r\times{\P^r}^\vee$ and denote by $\cP$ its closure (this is the so-called {\it conormal variety of $C\subset\P^r$}). Then the restrictions ($p_1$ and $p_2$) to $\cP$ of the canonical projections have images $C$ and $C^\vee$, respectively:
$$
\xymatrix{\P^r&\P(T_{\P^r})\ar[r]\ar[l]&{\P^r}^\vee\\
C\ar@{^{(}->}[]+<0ex,2.5ex>;[u]&\cP\ar@{^{(}->}[]+<0ex,2.5ex>;[u]%
\ar[r]^{p_2}\ar[l]_{p_1}&C^\vee\ar@{^{(}->}[]+<0ex,2.5ex>;[u]}
$$
Finally, let us recall that the Biduality (also called Reflexivity) theorem (cf. \cite[Section 1.3]{Tev}) states that $(C^{\vee})^\vee=C$, so that the diagram above is reversible, and we may assert that the general fiber of $p_2$ (the so-called {\it tangency locus} of a hyperplane) is a linear space. In particular one expects $p_2$ to be, indeed, birational for most projective varieties.
\begin{definition}\label{def:defect}
With the same notation as above, the number $e(C):=r-1-\dim(C^\vee)$ is called the  {\it dual defect} of $C$. If $e(C)>0$, we say that $C$ is  {\it dual defective}.
\end{definition}

It is well known (cf. \cite[Theorem 4.25]{Tev}) that the dual defect of a nonlinear smooth subvariety $C\subset \P^r$ is smaller than or equal to $\codim(C)-1$. We will need an extension of this result to mildly singular varieties, a topic that may be understood in the context of the theory of discriminant varieties of linear systems (see, for instance, \cite{LPS}). 

\begin{definition}
We say  that $C$ has only {\it nodal singularities}, if its normalization $M$ is smooth and the normalization morphism $\nu:M\to C$ is unramified, i.e.  $\nu^*\Omega_C \to \Omega_M$ is surjective. 
\end{definition}

We will assume from now on that $C$ has only  nodal singularities.
Denoting by $t:M \to \P(V)$ the composition of $\nu$ and the inclusion $C \to \P(V)$, and by  $\cN$ the cokernel of $dt$, we have an injection of vector bundles over $M$:
$$
\xymatrix{0\ar[r]& T_{M}\ar[r]^(.40){dt}&t^*T_{\P(V)}\ar[r]&\cN\ar[r]& 0}
$$

Let us denote $c:=\dim(C)$, and $e:=e(C)$.
The sheaf $\cN$ is then a vector bundle, of rank equal to $\codim(C,\P(V))=r-c$. 

Furthermore, setting $\cO_M(1):=t^*(\cO_{\P(V)}(1))$ and considering the pull-back to $M$ of the Euler sequence on the projective space $\P(V)$, we see that $\cN(-1)$ is globally generated by $V^\vee=H^0(\P(V),\cO_{\P(V)}(1))^\vee$. In particular we then have a well defined morphism:
\begin{equation}\label{eq:psi}
\psi:\P(\cN(-1))\to\P(V^\vee),
\end{equation}
whose image is precisely the dual variety $C^\vee$; furthermore, we have a commutative diagram:
\begin{equation}\label{eq:psi2}
\xymatrix{\P(\cN(-1))\ar@/^1.2pc/[rr]^{\psi}\ar[r]\ar[d]&\cP\ar[r]\ar[d]&\P(V^\vee)\\M\ar[r]&C&}
\end{equation}
where, again, $\cP$ is the conormal variety of $C$. Note that $\cP$ is a $\P^{r-c-1}$-bundle over a certain open set of $C$, isomorphic to its inverse image in $\P(\cN(-1))$.

\begin{proposition}\label{prop:split}
Let $C\subset\P(V)$ be a variety with only nodal singularities and assume that, with the same notation as above, the defect $e$ is bigger than zero. Then its normalization $M$ is swept out by a family of projective spaces $\P^e$ of $\cO_M(1)$-degree equal to one, whose general element satisfies that $(N_{\P^e,M})_{|\P^1}\cong E(0^{\frac{c-e}{2}},1^{\frac{c-e}{2}})$, for every line $\P^1\subset\P^e$.
\end{proposition}

\begin{proof}
The Biduality theorem tells us that, given a general point $h\in C^\vee$ (smooth point is enough), its inverse image in $\cP$ is isomorphic to $\P^e$, mapping one-to-one to a linear subspace contained in $C$. Its inverse image $F'_h$ into $\P(\cN(-1))$ is smooth and maps birationally to $\P^e\subset\cP$ (by the birationality of $\P(\cN(-1))\to \cP$). Consequently, it maps birationally also to $\P^e\subset C$. Moreover, it maps finite-to-one onto its image $F_h\subset M$ (because $\P(\cN(-1))\to \P(V^{\vee})$ is determined by its tautological line bundle $\cO_{\P(\cN(-1))}(1)$). Finally, since $t$ is finite, it follows that the composition
$$
F'_h\to F_h\to\P^e\subset C
$$
is finite and birational onto a smooth variety, hence it is an isomorphism. Thus $F_h\to\P^e$ is an isomorphism, too.

For the second part of the statement, note that we have short exact sequences on $F_h\cong F_h'$:
\begin{equation}\label{eq:ses1}
\xymatrix@=15pt{\cO_{F_h}\ar@{>->}[r]&\cN^\vee(1)_{|F_h}\ar@{->>}[r]&{T_{\P(\cN)|M}}_{|F'_h}\\
{T_{\P(\cN)|M}}_{|F'_h}\ar@{>->}[r]&N_{F'_h,\P(\cN)}=\cO_{F'_h}^{r-e-1}\ar@{->>}[r]&N_{F_h,M}\\
N_{F_h,M}(-1)\ar@{>->}[r]&\cO_{F_h}^{r-e}\ar@{->>}[r]&\cN(-1)_{|F_h}}
\end{equation}
The first is the restriction of the relative Euler sequence of $\P(\cN)$, the second comes from the differential of the natural map $\P(\cN_{|F_h})\to F_h$, and the third from the differential of the map $t:M \to \P(V)$. The three sequences
fit in the commutative diagram:
$$
\xymatrix@=30pt{\cO_{F_h}\ar@{>->}[r]\ar@{=}[d]&\cN^\vee(1)_{|F_h}\ar@{->>}[r]\ar@{>->}[d]&{T_{\P(\cN)|M}}_{|F'_h}\ar@{>->}[d]\\
          \cO_{F_h}\ar@{>->}[r]&\cO_{F_h}^{r-e}\ar@{->>}[r]\ar@{->>}[d]&\cO_{F_h}^{r-e-1}\ar@{->>}[d]\\
          &N_{F_h,M}^\vee(1)\ar[r]^{\cong}&N_{F_h,M}}
$$
We then have an isomorphism
\begin{equation}\label{eq:ses2}
N_{F_h,M}^\vee(1)\cong N_{F_h,M}.
\end{equation}

We consider now any line $\ell \subset F_h\subset M$. Since $N_{F_h,M}|_\ell \simeq E(a_1,\dots,a_{c-e})$,  with $a_1\leq\dots\leq a_{c-e}$, we note first that, since $N_{F_h,M}$ is nef, $a_1\geq 0$. Then the isomorphism (\ref{eq:ses2}) tells us that $a_{c-e}\leq 1$ and $\sum_i a_i=\frac{c-e}{2}$, so $N_{F_h,M}|_\ell \simeq E(0^{\frac{c-e}{2}},1^{\frac{c-e}{2}})$. This concludes the proof.
\end{proof}

\begin{corollary}\label{cor:embed}
Let $C\subset\P(V)$ be a variety with only nodal singularities, which  is not a linear subspace. Then the defect $e$ is smaller than or equal to $r-c-1$, that is $\dim(C)\leq\dim(C^\vee)$. Moreover, if equality holds, then   $t:M \to \P(V)$ is an embedding; in particular in this case $C$ is smooth.
\end{corollary}

\begin{proof}
As an immediate consequence of Proposition \ref{prop:split}, the defect $e$ is smaller than or equal to $c$, and $c\equiv e$ modulo $2$. If $e=c$, then $C\subset\P(V)$ would be a linear subspace, a contradiction.
Hence we may assert that $e\leq c-2$.

Denoting $h:=r-c-1-e$, this inequality reads as
$c-2\geq e=r-c-1-h$, that is $2c\geq r+1-h>r-h$. 

Now, by a corollary of Fulton-Hansen Theorem (cf. \cite[Theorem 4.30]{Tev}, or \cite[Theorem~3.4.1]{L2}), if $h\leq 0$ it follows that  $t$ is an embedding, so $C$ is smooth, which in turn implies, by Zak's Theorem on Tangencies (cf. \cite{Zak}, \cite[Theorem 4.25]{Tev}), that $\dim(C)\leq \dim(C^\vee)$, that is $h\geq 0$.
\end{proof}

%

\begin{corollary}\label{cor:smooth}
Let $C\subset\P(V)$ be a variety with only nodal singularities. Assume that the morphism $\psi$ defined in (\ref{eq:psi}), considered as a surjective morphism onto its image $C^\vee$, is equidimensional. Then the normalization $\tl{C}^\vee$ of $C^\vee$ is smooth. 
\end{corollary}

\begin{proof}
We will consider the factorization $\tl{\psi}:\P(\cN(-1))\to \tl{C}^\vee$ of $\psi$ onto the normalization $\tl{C}^\vee$ of $C^\vee$. Since $\psi$ is equidimensional by hypothesis, so is $\tl{\psi}$. Moreover, as in the proof of Proposition \ref{prop:split}, the general fiber of $\psi:\P(\cN(-1))\to C^\vee$ is isomorphic to $\P^e$, hence the same holds for the fibers of $\tl{\psi}$. Finally the pull-back of $\cO_C(1)$ provides a $\tl{\psi}$-ample line bundle on $\P(\cN(-1))$ having degree one on the general fiber of the morphism, hence it follows by \cite[Lemma 2.12]{Fuj1} that $\tl{\psi}$ is a $\P^e$-bundle and, in particular $\tl{C}^\vee$ is smooth. 
\end{proof}


\section{Preliminaries on rational curves on Fano manifolds}\label{sec:ratcurves}

Let $X$ be a Fano manifold of Picard number one and dimension $m$, defined over the field of complex numbers. In this section we will 
briefly review, in our setting, some well known results on deformations of rational curves on $X$. Most of the listed results work on broader settings: we refer to \cite{Hw}, \cite{KS} and \cite{kollar} for more details.

A {\it family of rational curves on }$X$ is, by definition, the normalization $\cM$ of an irreducible component of the scheme $\rat^n(X)$. Each of these families comes equipped with a smooth $\P^1$-fibration $p:\cU\to \cM$ and an evaluation morphism $q:\cU\to X$. Given a point $x\in q(\cU)$, we denote by $\cM_x$ the normalization of the set $p(q^ {-1}(x))$, and by $\cU_x$ the normalization of its fiber product with $\cU$ over $\cM$. We say that the family $\cM$ is:
\begin{itemize}
\item {\it dominating} if $q$ is dominant,
\item {\it locally unsplit} if $\cM_x$ is proper for general $x\in X$,
\item {\it unsplit} if $\cM$ is proper.
\end{itemize}
Note that $\rat^n(X)$ is quasi-projective, hence the properness of $\cM_x$ or $\cM$ implies their projectivity. The anticanonical degree $-K_X\cdot\Gamma$ is the same for every curve $\Gamma$ of the family $\cM$, thus we will denote it by $-K_X\cdot\cM$; moreover, we will denote by $c(\cM)$, or simply $c$ if there is no possible confusion, the number $c:=-K_X\cdot\cM-2$.


\medskip


\begin{proposition}\label{prop:RCbasic} Let $X$ be a Fano manifold of Picard number one,  $\cM$ be a locally unsplit dominating family of rational curves, and set $c=-K_X\cdot \cM-2$. With the same notation as above, we have the following:
\begin{itemize}
\item[(1)] The variety $\cM$ has dimension $m+c-1$ and, for a general point $x \in X$, ${\cM}_x$ is a disjoint union of smooth projective varieties of dimension $c$.
\item[(2)] There exists a nonempty open set $X_0$ of $X$ satisfying that, for every $x\in X_0$, the normalization $f:\P^1\to \Gamma$ of any element $\Gamma$ of $\cM$ passing by $x$ is {\emph{free}}, that is $f^*T_X$ is a nef vector bundle. 
\item[(3)]  For a general $x\in X$, $p_{|\cU_x}:\cU_x\to \cM_x$ is a $\P^1$-bundle, admitting a section $\sigma_x$ whose image lies in $q^{-1}(x)$.
\item[(4)] A rational curve $\Gamma$ given by a general member of $\cM$, is {\em standard}, i.e., denoting by $f: \P^1 \to X$ its normalization, $f^{\ast}T_X 
\cong E(2,1^c,0^{m-c-1})$.
\end{itemize}
\end{proposition}

\begin{proof}
For the first and second statement, see \cite[II. 1.7~and~3.11]{kollar}. (3) follows by \cite[II. Theorem~2.12]{kollar}, and by the fact that the point $x$ determines a section of $p_{|p^{-1}(\cM_x)}:p^{-1}(\cM_x)\to \cM_x$: in fact, \cite[Theorem~3.3]{Ke2} implies that $q^{-1}(x)$ consists of a finite set of points and a unique component mapping one to one onto $\cM_x$ via $p$. Finally (4)
follows from \cite[IV. Corollary~2.9]{kollar}.
\end{proof}

\begin{definition}\label{def:uniform}
With the same notation as above, given the normalization $f:\P^1\to X$ of an element $\Gamma$ of $\cM$, we say that $\Gamma$ has {\it splitting type} $(a_1^{k_1},\dots,a_r^{k_r})$ on $X$, with $k_1+\dots +k_r=\dim(X)$, $a_1\leq\dots \leq a_r$, if $f^*T_X\cong E(a_1^{k_1},\dots,a_r^{k_r})$. 
We then say that $\cM$ is {\it uniform} (resp. {\it locally uniform}) if every curve $\Gamma$ of $\cM$ (resp. every curve of $\cM_x$, $x\in X$ general)  has the same splitting type.
\end{definition}

The following lemma is well known.

\begin{lemma}\label{lem:unifsmooth}
Let $X$ be a smooth variety and let $\cM$ be a locally unsplit dominating family of rational curves in $X$. If $T_X$ is nef on every curve of the family $\cM$ (in particular, if $\cM$ is uniform), then the evaluation morphism $q:\cU\to X$ is smooth. 
\end{lemma}

\begin{proof}
The nefness condition implies that 
the differential $dq$ has constant maximal rank and, in particular, that the morphism $q$ is equidimensional. But since $X$ is smooth, it follows that $q:\cU\to X$ is also flat (cf. \cite[Theorem 23.1]{Ma}), hence smooth. 
\end{proof}


\subsection{Variety of minimal rational tangents of a family $\cM$}\label{ssec:VMRT}

With the same notation as above, if $\cM$ is a locally unsplit dominating family, and $x\in X$ is a general point, then $\cM_x$ is smooth, and (3) from Proposition \ref{prop:RCbasic} allows us to claim that there exists a rational map $\tau_x$ from $\cM_x$ to 
$\P(\Omega_{X,x})$, called the {\it tangent map of $\cM$ at $x$}, sending the general element of $\cM_x$ to its tangent direction at $x$ (see Remark \ref{rem:O1} below). It is known (cf. \cite[Theorem~3.4]{Ke2}, \cite[Theorem~1]{HM2}) that $\tau_x$ is a finite  and birational morphism (hence it is the normalization of its image) onto a variety $\cC_x$ usually called the {\it variety of minimal rational tangents} (VMRT, for short) of $\cM$ at $x$. The closure of the union of all the $\cC_x$, $x\in X$ general, into $\P(\Omega_X)$, 
is denoted by $\cC$. 

\begin{remark}\label{rem:O1} Given a general point $x\in X$, the map $\tau_x$ may be understood as follows. Let us consider the section $\sigma_x:\cM_x\to \cU_x$ defined in Proposition \ref{prop:RCbasic} (3), and let $K_{\cU_x/\cM_x}$ denote the relative canonical divisor of $\cU_x \to \cM_x$. Via the composition of the morphisms $T_{\cU_x/\cM_x}|_{q^{-1}(x)} \hookrightarrow T_{\cU_x}|_{q^{-1}(x)} $ and $T_{\cU_x}|_{q^{-1}(x)} \to T_{X,x}\otimes \cO_{q^{-1}(x)}$, it follows from \cite[Theorem~3.3]{Ke2} that $T_{\cU_x/\cM_x}|_{q^{-1}(x)} $ is a subbundle of $T_{X,x}\otimes \cO_{q^{-1}(x)}$. Restricting this inclusion of vector bundles to $\cM_x$ via $\sigma_x$ yields a morphism $\cM_x \to \P(\Omega_{X,x})$, which is nothing but the tangent map $\tau_x$. Hence the tangent map $\tau_x$ satisfies
\begin{equation}
\tau_x^{\ast}\cO_{\P(\Omega_{X,x})}(1)=\cO_{\cM_x}(\sigma_x^*K_{\cU_x/\cM_x}).
\end{equation}
In particular, this line bundle 
is ample and globally generated.
\end{remark}

\begin{proposition}\cite[Proposition~2.7]{Ar}\label{prop:imm} With the same notation as above, given a locally unsplit dominating family of rational curves $\cM$ on $X$, and a general point $x\in X$,
$\tau_x$ is immersive at $[\Gamma] \in \cM_x$ if and only if $\Gamma$ is standard.
\end{proposition}

%
%

%
%
%
%

Note that it is well known that for a rational homogeneous manifold $S=G/P$ ($G$ semisimple, $P$ parabolic subgroup) of Picard number one, the generator $A$ of the Picard group is very ample (Theorem of Borel--Weil, see \cite{JPS}) and there exists an unsplit dominating family $\cM_S$ of curves of degree one with respect to $A$. One of the key results we will use tells us that, under certain conditions, one may decide whether a Fano manifold of Picard number one is homogeneous by comparing the VMRT of one of its locally unsplit dominating families of rational curves with the VMRT of $\cM_S$. 
\begin{theorem}\label{them:HH} Let $X$ be a Fano manifold of Picard number one, $\cM$ be a locally unsplit dominating family of rational curves on $X$, and $x$ be a general point in $X$. Let $S=G/P$ be a rational homogeneous manifold of Picard number one as above, and $o\in S$ be any point. Assume moreover that $P$ is a parabolic subgroup corresponding to a long simple root of $G$. If the VMRT $\cC_x\subset \P(T_{X,x}^\vee)$ of $\cM$ at $x\in X$ is projectively equivalent to the VMRT $\cC_o \subset \P(T_{S,o}^\vee)$ of $\cM_S$ at $o$, then $X$ is isomorphic to $S$.
\end{theorem}

This result was proven by Mok (see \cite{Mk3}) for Hermitian symmetric spaces and homogeneous contact manifolds, and then generalized by Hwang and Hong, who stated Theorem \ref{them:HH} in \cite{HH}.


\section{Minimal sections of $\P(T_X)$ over rational curves}\label{sec:minsec}
Along this section, $X$ will denote a Fano manifold of Picard number one and dimension $m$, not isomorphic to a projective space, and $\cM$ a locally unsplit dominating family of rational curves in $X$, of anticanonical degree $c+2$. We will use for them the notations introduced in Section \ref{sec:ratcurves}. Furthermore, we will denote by $\phi:\P(T_X)\to X$ the natural projection from the Grothendieck projectivization of $T_X$ onto $X$. 


\begin{definition}\label{def:minsec1}
Given a free element $\Gamma$ of $\cM$, with normalization $f:\P^1\to X$, a section of $\P(T_X)$ over $\Gamma$  corresponding to a surjective map $f^*T_X\to \cO_{\P^1}$ will be called {\it minimal} and the corresponding rational curve will be denoted by $\overline{\Gamma}$. 
 \end{definition}
 
\begin{remark}
Minimal sections as in the above definition exist whenever $X$ is different from the projective space, by the Main Theorem of \cite{CMSB}.
\end{remark}

We may now consider a nonempty open subset $\cM_0\subset\cM$ parametrizing standard curves, and denote by $p_0:\cU_0\to \cM_0$ and $q_0:\cU_0\to X$ the corresponding family and an evaluation morphism respectively. Then $p^{}_{0*}q_0^*\Omega_X$ and $p_0^*p^{}_{0*}q_0^*\Omega_X$ are rank $m-1-c$ vector bundles over $\cM_0$ and $\cU_0$, respectively, and the natural morphism $\ol{p}_0:\ol{\cU}_0:=\P(p_0^*(p^{}_{0*}q_0^*\Omega_X)^\vee)\to \ol{\cM}_0:=\P((p^{}_{0*}q_0^*\Omega_X)^\vee)$ is a smooth $\P^1$-fibration. Together with the natural morphism $\ol{q}_0:\ol{\cU}_0\to\P(T_X)$, this data provides a family of rational curves in $\P(T_X)$, and hence a morphism from $\ol{\cM}_0$ to $\rat^n(\P(T_X))$, which is injective. Moreover, by construction, the image of every element of $\ol{\cM}_0$ in $\rat^n(\P(T_X))$ corresponds to a minimal section of $\P(T_X)$ over a curve of $\cM_0$. Let us denote by $\ol{\cM}$ the normalization of the closure of its image, and by $\ol{p}:\ol{\cU}\to\ol{\cM}$, $\ol{q}:\ol{\cU}\to\P(T_X)$ the corresponding universal family and evaluation morphism. 
Furthermore, as we will see later, $\ol{\cM}$ is, in fact, the normalization of an irreducible component of $\rat^n(\P(T_X))$.

\begin{definition}\label{def:minsec2}
With the same notation as above, given a locally unsplit dominating family $\cM$, the family of rational curves parametrized by $\ol{\cM}$ is called the {\it family of minimal sections of $\P(T_X)$ over curves of $\cM$}.
\end{definition}

Moreover, we have natural morphisms $\phi_\cM:\ol{\cM}\to\cM$, $\phi_\cU:\ol{\cU}\to\cU$, 
fitting in the following commutative diagram: 
\begin{equation}\label{eq:curvessec}
\xymatrix@=35pt{\ol{\cM}\ar[d]^{\phi_\cM}&\ol{\cU}\ar[d]^{\phi_\cU}\ar[l]_{\ol{p}}\ar[r]^{\ol{q}}&\P(T_X)\ar[d]^\phi\\{\cM}&{\cU}\ar[l]_p\ar[r]^q&X}
\end{equation}
Note that Proposition \ref{prop:RCbasic} (4) implies that  the fibers of $\phi_\cM$ over every standard curve of $\cM_0\subset \cM$ are isomorphic to $\P^{m-c -2}$, so $\overline{\cM} $ has dimension $2m-3$. However, the image of $\ol{q}$, that we will denote by $D$, may have dimension smaller than $\dim(\overline{\cU})=2m-2$.

\begin{definition}\label{def:2-uniform}
We say that $\cM$ is $2$-{\it uniform} if both $\cM$ and $\overline{\cM}$ are uniform.
\end{definition}


\subsection{Projective duality for VMRT's}\label{ssec:dualVMRT}

We will study now the relation between the minimal sections of $\P(T_X)$ over rational curves of a locally unsplit dominating family $\cM$, and its VMRT's.
We first recall that $\P(T_X)$ supports a {\em contact structure} $\cF$, defined as the kernel of the composition of the differential of the natural projection $\phi:\cX\to X$ with the co-unit map
$$
\theta: T_{\cX}\stackrel{d\phi}{\longrightarrow}\phi^*T_X=\phi^*\phi_*\cO(1)\longrightarrow\cO(1).
$$
Note that $\theta$ fits in the following commutative diagram, with exact rows and columns:
\begin{equation}\label{eq:contact}
\xymatrix@=35pt{T_{\P(T_X)/X}\ar@{>->}[r]\ar@{=}[d]&\cF\ar@{->>}[r]\ar@{>->}[]+<0ex,-2ex>;[d]&\Omega_{\P(T_X)/X}(1)\ar@{>->}[d]\\
          T_{\P(T_X)/X}\ar@{>->}[r]&T_{\cX}\ar@{->>}[r]\ar@{->>}[d]^{\theta}&\phi^*T_X\ar@{->>}[d]\\
          &\cO(1)\ar@{=}[r]&\cO(1)}
\end{equation}

The distribution $\cF$ being contact means precisely that it is maximally non integrable, i.e. that the morphism $d\theta:\cF\otimes\cF\to T_X/\cF\cong \cO(1)$ induced by the Lie bracket is everywhere non-degenerate. This fact can be shown locally analytically, by considering, around every point, local coordinates $(x_1,\dots,x_m)$ and vector fields $(\zeta_1,\dots,\zeta_m)$, satisfying $\zeta_i(x_j)=\delta_{ij}$. Then the contact structure is determined, around that point, by the $1$-form $\sum_{i=1}^m\zeta_idx_i$ (see \cite{KPSW} for details).

The next proposition describes the infinitesimal deformations of a general minimal section $\overline{\Gamma} $.
\begin{proposition}\label{prop:splittype}
Let $\overline{f} :\P^1\to \P(T_X)$ denote the normalization of a minimal section $\overline{\Gamma} $ of $\P(T_X)$ over a standard rational curve in the class $\Gamma $. Then $\overline{\cM} $ is the normalization of an irreducible component of $\rat^n(\P(T_X))$, smooth at $\overline{\Gamma} $, of dimension $2m-3$, and
 $$
\overline{f}^*T_{\P(T_X)}\cong E\big(-2,2,(-1)^{e },1^{e },0^{2m-3-2e }\big),\mbox{ for some }e \leq c .
$$
\end{proposition}

\begin{proof}
Writing ${f}^*T_X=E(2,1^{c },0^{m-c -1})$ and taking into account that $\overline{f}^*\cO(1)=\cO$, the relative Euler sequence of $\P(T_X)$ over $X$, pulled-back via $\overline{f} $ provides $\overline{f}^*T_{\cX/X}=E(-2,(-1)^{c },0^{m-c -2})$. The upper exact row of diagram (\ref{eq:contact}) provides:
$$
0\to E(-2,(-1)^{c },0^{m-c -2})\longrightarrow \overline{f}^*\cF\longrightarrow E(2,1^{c },0^{m-c -2})\to 0.
$$
On the other hand, $\overline{f}^*\cO(1)=\cO$ also implies that $d\overline{f} :T_{\P^1}\to \overline{f}^*T_{\cX}$ factors via $\overline{f}^*\cF$, hence this bundle has a direct summand of the form $\cO(2)$. Being $\cF$ a contact structure, it follows that $\overline{f}^*\cF\cong\overline{f}^*\cF^\vee$, so this bundle has a direct summand $\cO(-2)$, as well.

From this we may already conclude that
\begin{equation}\label{eq:splitcontact}
\overline{f}^*\cF\cong E(-2,2,(-1)^{e },1^{e },0^{2m-2e -4}), \mbox{ for some }e \leq c ,
\end{equation}
hence the bundle $\overline{f}^*T_\cX$ is isomorphic either to
$E(-2,2,(-1)^{e },1^{e },0^{2m-2e -3})$ or to $E(2,(-1)^{e +2},1^{e },0^{2m-2e -4})$.
On the other hand, the fact that $\dim\overline{\cM} =2m-3$ implies that $h^0(\P^1,\overline{f}^*T_\cX)\geq 2m$, which allows us to discard the second option. 

Finally, in the first case $2m=h^0(\P^1,\overline{f}^*T_\cX)\geq\dim_{[\overline{f} ]}\Hom(\P^1,\cX)\geq \dim\overline{\cM}+3\geq 2m$ and, in particular, $\Hom(\P^1,\cX)$ is smooth at $[\overline{f} ]$. Quotienting by the action of $\Aut(\P^1)$, this also tells us that $\overline{\cM} $ is a component of $\rat^n(\P(T_X))$, smooth at $\overline{\Gamma} $.
\end{proof}

\begin{definition}\label{def:defectcurve}
Given a minimal section $\overline{\Gamma} $ over a standard curve $\Gamma\in\cM$, the number $e$ provided by Proposition \ref{prop:splittype} will be called the {\it defect of $\overline{\cM} $ at}
$\overline{\Gamma} $.
\end{definition}

\begin{remark}\label{rem:corank}
The defect of $e$ of $\overline{\cM} $ at $\overline{\Gamma} $ can be interpreted as the corank of $\ol{q}$ at any point of $\ol{\Gamma}$, minus one. In fact, let us denote by $\overline{f} :\P^1\to \cX$ the normalization of $\overline{\Gamma}$. 
Consider the evaluation morphism $\ev:\Hom_{[\overline{f}]}(\P^1,\cX)\times \P^1\to \cX$, where $\Hom_{[\overline{f}]}(\P^1,\cX)$ stands for the irreducible component of $\Hom(\P^1,\cX)$ containing $[\overline{f}]$ (which is smooth at $[\overline{f}]$ and of dimension $2m$ by Proposition \ref{prop:splittype}). 
The morphism $\ol{q}$ is obtained from this morphism by quotienting by the action of the group of automorphisms of $\P^1$, hence, using the description of this differential provided in \cite[II. Proposition~3.4]{kollar}, it follows that 
the cokernel of the evaluation of global sections $H^0(\P^1,\overline{f}^*(T_\cX))\otimes\cO_{\P^1}\to\overline{f}^*(T_\cX)$ is isomorphic to $E(-2, (-1)^{e})$, and hence it has rank equal to $e+1$. This implies that the corank of $\ol{q}$ at any point of ${\ol{p}}^ {-1}(\Gamma)$ is equal to $e+1$. 
\end{remark}

The next result establishes the relation between $D=\im \ol{q}\subset\cX$
and the VMRT $\cC\subset \P(\Omega_X)$, at the general point $x$. This was first obtained in \cite[Corollary 2.2]{HR} in their study of the moduli space of stable vector bundles on a curve. Our line of argumentation here is based on the proof of \cite[Proposition~1.4]{Hw}. 

\begin{proposition}\label{prop:dual}
Being $x\in X$ general, assume that $\cC_x$ is irreducible. Then $D_{x}:=D\cap \P(T_{X,x})\subset \P(T_{X,x})$ is the dual variety of $\cC_{x}\subset \P(\Omega_{X,x})$.
\end{proposition}

%

\begin{proof} 
Let 
$f :\P^1\to X$ be the normalization of a standard element $\Gamma\in\cM $, and $O\in\P^1$ be a point satisfying $f (O)=x$. 
By Proposition \ref{prop:imm}, the tangent map $\tau_{x}:\cM_{x}\to\cC_{x}\subset\P(\Omega_{X,x})$ is immersive at $[\Gamma]$ and we may use it to describe the tangent space of $\cC_{x}$ at $P:=\tau_{x}(\Gamma )$.

Note that we have a filtration $T_{X,x}\supset V_1(f )\supset V_2(f )$, where $V_1(f )$ and $V_2(f )$ correspond, respectively, to the fibers over $O$ of the (unique) subbundles of $f^*T_X$ isomorphic to $E(2,1^{c})$ and $E(2)$. Quotienting by homotheties in $T_{X,x}$ (in order to obtain the Grothendieck projectivization $\P(\Omega_{X,x})$), the subspace $V_2(f)$ provides the point $P$, and it is known that $V_1(f)$ provides the projective tangent space of $\cC_x$ at $P$ (cf. \cite[Lemma~2.1]{AW}, \cite[Proposition~1.4]{Hw}). But then 
the set of hyperplanes in $\P(\Omega_{X,x})$ containing this tangent space, is in one-to-one correspondence (via restriction from $\P(f^*T_X)$ to $\P(T_{X,x})$) with surjective maps $f^*T_X\to\cO_{\P^1}$, that correspond to minimal sections of $\P(T_X)$ over $\Gamma$. 

\end{proof}

\begin{corollary}\label{cor:dualdef1s}
Being $x\in X$ general, assume that $\cC_x$ is irreducible, and let $\overline{\Gamma}$ be a general minimal section of $\cX$ over a standard element of $\cM_{x}$. Then the dual defect of $\cC_{x}$ equals the defect of $\overline{\cM} $ at $\overline{\Gamma}$.
\end{corollary}

\begin{proof}
For general $x$, the dual defect of $\cC_{x}$ equals $\codim(D\subset\cX)-1$, by Proposition \ref{prop:dual}. This number is the corank of $\ol{q}$ at the general point minus one, which is equal to the defect $e$ of $\overline{\cM} $ at $\overline{\Gamma}$ by Remark \ref{rem:corank}.
%
\end{proof}


\section{Proofs}\label{sec:proofs}

\noindent{\it Proof of Theorem \ref{thm:main}}\quad We note first that $\cC_x$ is irreducible. In fact, since by hypothesis $\cM$ is uniform, locally unsplit and dominating, then it follows by Lemma \ref{lem:unifsmooth} that the evaluation morphism $q:\cU\to X$ is smooth. Since, moreover, $\cM$ is unsplit, then $\cU$ is projective, and $q$ is surjective. We may then consider the Stein factorization of $q$, which provides a smooth finite morphism $q':\Spec_X(q_*\cO_\cU)\to X$. Since $X$ is Fano, hence simply-connected, it follows that $q'$ is an isomorphism, which finally tells us that the fibers of $q$ are connected and smooth, thus irreducible.

We now claim that $D_x\subset\P(T_{X,x})$ has only nodal singularities, for the general point $x$. Since we are assuming that $\cM$ is $2$-uniform, then the corank of $\ol{q}$ is constantly equal to $e+1$, by Remark \ref{rem:corank}, and $\ol{q}$ is equidimensional. At this point, we consider diagram (\ref{eq:curvessec}) and restrict its right hand side to $\cM_x$ via the section $\sigma_x$ provided by Proposition \ref{prop:RCbasic} (3). We already know that $\cC_x$ has only nodal singularities by Proposition \ref{prop:imm}; then our description of the dual variety of $\cC_x$ given in Proposition \ref{prop:dual}, tells us that $\phi_\cU^{-1}(\cM_x)\cong\P(\cN(-1))$, where $\cN$ stands for the cokernel of the inclusion $d\tau_x:T_{\cM_x}\to T_{\P(\Omega_{X,x})}$, and $\ol{q}\big(\phi_\cU^{-1}(\cM_x)\big)=D_x$. Then the equidimensionality of $\ol{q}$ tells us that the normalization of $D_x$ is smooth by Corollary \ref{cor:smooth}. Finally, since $\ol{q}$ has constant rank, it follows that the normalization morphism of $D_x$ is unramified, as we claimed.

If  $\cC_x$ is a linear subspace, then $X$ is a projective space (cf. \cite[Proposition~5]{Hw2}), and our hypothesis on $-K_X\cdot\cM$  implies, in this case, that $X\cong \P^1$, that is $(\cD,i)=({\rm{A}}_1, 1)$. Therefore we may assume that neither $\cC_x$ nor $D_x$ are linear subspaces, and we may apply Corollary \ref{cor:embed} to both, obtaining that $\dim(\cC_x)\leq \dim(D_x)$, and $\dim(D_x)\leq \dim(\cC_x)$. Then equality holds and the second part of Corollary \ref{cor:embed} tells us that $\cC_x$ and $D_x$ are both smooth. 

Finally, since we are assuming that $c \leq 2(m-1)/3$, we may apply \cite[Theorem~4.5]{Ein1} to get that $\cC_x \subset \P(\Omega_{X,x})$ is one of the following:
\begin{itemize}
  \item a hypersurface in $\P^2$ or $\P^3$,
  \item the Segre embedding of $\P^1 \times \P^{c-1} \subset \P^{2c-1}$,
  \item the Pl\"ucker embedding of $\G(1,4)$, 
  \item the Spinor variety $S_{4} \subset \P^{15}$. 
\end{itemize}
Since the only smooth nonlinear hypersurfaces with smooth dual have degree two (see, for instance, \cite[10.2]{Tev}), in the first case $\cC_x$ is a conic or a two-dimensional quadric, hence in all cases the listed varieties are projectively equivalent
to the VMRT's of the homogeneous manifolds in the statement, and we may conclude by Theorem \ref{them:HH}.
\par\bigskip
\noindent{\it Proof of Corollary \ref{cor:main}}\quad
Proceeding as in the proof of Theorem \ref{thm:main}, we get that $\cC_x$ and $D_x$ are smooth for the general $x\in X$, and that they have the same dimension. Being $\cC_x$ a nonlinear complete intersection, it follows by  \cite[Theorem 5.11]{Tev} that $D_x$ is a hypersurface, hence $\cC_x$ is a hypersurface, too, and arguing as in the last part of the proof of Theorem \ref{thm:main},  it follows that $\cC_x$ is a smooth hyperquadric, which allows to conclude by Theorem \ref{them:HH}.


\par\bigskip
\noindent{\it Proof of Corollary \ref{cor:e-ample}}\quad
Let $p:\cU\to\cM$ denote the uniform unsplit dominating family of rational curves on $X$, and consider, with the same notation as above, the corresponding family $\ol{p}:\ol{\cU}\to\ol{\cM}$ of minimal sections of $\P(T_X)$ over curves of $\cM$. It is enough to check 
that $\ol{\cM}$ is 
uniform. Let us then denote by $e$ the defect of $\ol{\cM}$ at a general curve $\ol{\Gamma}$.
Note that the contraction $\epsilon:\cX\to \cY$ determined by $\cO_{\P(T_X)}$ can be obtained by as the quotient modulo homotheties of the contraction:
$$
\Spec_X\left(\bigoplus_{r\geq 0} S^rT_X\right)\lra
\Spec\left(\bigoplus_{r\geq 0} H^0(X,S^rT_X)\right)
$$
It is known that this map is a symplectic contraction, hence it is in particular {\it semismall}  (see \cite{1-ample}, \cite{Kal}), and this property is then inherited by our contraction $\epsilon$. It follows then that, for every irreducible closed subset $E\subset\cX$, the general fiber of the restriction $\epsilon_{|E}$ has dimension at most equal to the codimension of $E$ in $\cX$. 

Now, let $E$ be an irreducible component of $\Exc(\epsilon)$ containing $D=\ol{q}(\ol{\cU})$, $L$  the locus of curves of the family $\ol{\cM}$ passing by a general point $\ol{x}\in D$, and $F$ the fiber of the restriction $\epsilon:E\to\epsilon(E)$ passing through $\ol{x}$. Then we have:
$$
\codim(D\subset\cX)\geq\codim(E\subset\cX)\geq\dim(F)\geq\dim(L).
$$
Corollary \ref{cor:dualdef1s} tells us that $e$ is equal to $\codim(D\subset\cX)-1$, and to the dimension of the general fiber of $\ol{q}:\cU\to D$, which implies that $\dim(L)=e+1$. Summing up, we get that $D=E$, and that $F=L$. 

Moreover $F$ is an $(e+1)$-dimensional smooth variety swept out by curves of the family $\ol{\cM}$, which is unsplit, and the dimension of the subfamily of this curves passing by the general point of $F$ is $e$. It then follows by \cite{CMSB} that $F\cong\P^{e+1}$, and curves of $\ol{\cM}$ contained in $F$ are lines. Furthermore, we claim that the smoothness of the map $\epsilon_{|E}:E\to \epsilon(E)$ implies that every fiber $F$ contained in $E$ is necessarily isomorphic to $\P^{e+1}$. In fact, this follows from a classical result of Hirzebruch-Kodaira \cite{HK} (as noted in \cite[Theorem 1']{HM1}).

In particular, this implies that every curve $\ol{\Gamma}$ of $\ol{\cM}$ is a line on a fiber $F\cong\P^{e+1}$ contained in $E$, and so, denoting by $\ol{f}$ the normalization $\ol{\Gamma}$, $\ol{f}^*T_{\cX}$ contains a vector subbundle of the form $E(2,1^e)$. On the other hand, by Proposition \ref{prop:splittype} and semicontinuity, the splitting type of $T_{\cX}$ at $\ol{\Gamma}$ is of the form $(-2,2,(-1)^{e'},1^{e'},0^{2m-3-2e'})$ for $e'\leq e$. We may then conclude that $e'=e$.

\begin{acknowledgements}
The results in this paper were obtained mostly while the second author was a Visiting Researcher at the Korea Institute for Advanced Study (KIAS) and the Department of Mathematics of the University of Warsaw. He would like to thank both institutions for their support and hospitality. The authors would like to thank J. Wi\'sniewski for his interesting comments and discussions on this topic, and an anonymous referee, whose remarks helped to improve substantially the final form of this paper.
\end{acknowledgements}

\bibliographystyle{spmpsci}      



\end{document}